\documentclass[letterpaper, 10 pt, conference]{ieeeconf}  
\IEEEoverridecommandlockouts                              

\usepackage{amsmath}
\usepackage{amsfonts}
\usepackage{latexsym}
\usepackage{amssymb}
\usepackage{graphicx}
\usepackage{color}
\usepackage{chemarrow}
\usepackage{fullpage}
\usepackage{framed}
\usepackage{verbatim}
\usepackage{epsfig}
\usepackage{hyperref}
\usepackage{authblk}

\title{On the Lyapunov function for complex-balanced mass-action systems.}
\author{\small Manoj Gopalkrishnan\thanks{manojg@tifr.res.in}}
\affil{School of Technology and Computer Science,\\ Tata Institute of Fundamental Research, Mumbai, India.}
\begin{document}
\maketitle

\newtheorem{theorem}{\bf{Theorem}}[section]
\newtheorem{lemma}[theorem]{\bf{Lemma}}
\newtheorem{corollary}[theorem]{\bf{Corollary}}
\newtheorem{proofl}[theorem]{\bf{Proof}}
\newtheorem{open}{\bf{Open}}
\newtheorem{remark}[theorem]{Remark}
\newtheorem{definition}{\bf{Definition}}[section]
\newtheorem{conjecture}{\bf{Conjecture}}[section]
\newtheorem{postulate}{\bf{Postulate}}[section]

\newtheorem{example}{\bf{Example}}[section]
\newtheorem{thought}{\bf{Thought Experiment}}[section]
\newtheorem{notation}[definition]{\bf{Notation}}
\newtheorem{fact}{\bf{Fact}}
\newtheorem{note}[definition]{\bf{Note}}
\newcommand{\CE}{{\cal E}}
\newcommand{\MC}{\mathbb{C}}
\newcommand{\MR}{\mathbb{R}}
\newcommand{\MZ}{\mathbb{Z}}
\newcommand{\MN}{\mathbb{N}}
\newcommand{\MQ}{\mathbb{Q}}
\newcommand{\iton}{i=1,2,\ldots,n}
\newcommand{\jtom}{j=1,2,\ldots, m}
\newcommand{\MP}{\mathbb{R}_{>0}}
\newcommand{\bs}{\boldsymbol}
\newcommand{\op}{\operatorname}

\newcommand{\ra}{\rightarrow}
\def\R{\mathcal R}
\def\C{\mathcal C}
\renewcommand\comment[1]{{\color{magenta} $\star$#1$\star$}}
\newenvironment{bullets}%
        {\begin{list}
                {\noindent\makebox[0mm][r]{$\bullet$}}
                {\leftmargin=5.5ex \usecounter{enumi}
 		 \topsep=1.5mm \itemsep=-.75ex}
        }
        {\end{list}}
       
\begin{center}

\end{center}

\begin{abstract}
We present a new proof, using the log-sum inequality, that the pseudo-Helmholtz free energy function is a Lyapunov function for complex-balanced mass-action systems. This proof is shorter and simpler than previous proofs.
\end{abstract}

\section{Introduction}
Complex-balanced mass-action systems were introduced by Horn and Jackson in 1972~\cite{HornJackson}. They have also been called Toric Dynamical Systems~\cite{TDS} because of connections to toric geometry. Horn and Jackson showed that such systems admit Lyapunov functions of the form $g_\alpha(x) := \sum_i x_i \log x_i - x_i - x_i \log \alpha_i$, where the sum is over all species, the variables $x_i$ represent concentrations of species, and $\alpha$ represents a special concentration vector, called a point of complex balance. We present a new proof that makes use of the log-sum inequality and graph-theoretic ideas to help shorten their proof.

\section{Preliminaries}
Fix a finite set $S$ of \textbf{species}. A \textbf{chemical complex} (or an $S$-complex) is a vector $y\in\MN^S$. A \textbf{chemical reaction} (over $S$) is a pair $(y,y')$ of chemical complexes, written $y\ra y'$, with {\em reactant} $y$ and {\em product} $y'$. A \textbf{chemical reaction network} consists of a finite set $S$ of species, and a finite set $\R$ of chemical reactions. 

When $a,b$ are vectors in $\MR^S$ then $a/b$ is shorthand for the vector $(a_i/b_i)_{i\in S}$, the notation $a^b$ is shorthand for the monomial $\prod_{i\in S} a_i^{b_i}$, and 
$\log a$ is shorthand for the coordinate-wise logarithm $(\log a_i)_{i\in S}$.

A \textbf{mass action system} consists of a chemical reaction network $(S,\R)$ and a \textbf{rate function} $k:\R\to\MR_{>0}$. The \textbf{mass-action equations} are the system of ordinary differential equations in {\em concentration} variables $\{x_i(t) \mid i\in S\}$: 
\begin{equation}\label{eqn:ma}
\dot{x}(t) = \sum_{y\to y' \in \R} k_{y\to y'} (y' - y) x(t)^y 
\end{equation}
where $x(t)$ represents the vector $(x_i(t))_{i\in S}$ of concentrations at time $t$.

Let $\C =\{ y, y' \mid  y\to y'\in\R\}$ be all the complexes that occur in $\R$. Then the directed graph with vertex set the complexes $\C$ and edge set the reactions $\R$ is called the \textbf{reaction graph}. 

Given a concentration vector $x\in\MR^S_{\geq 0}$, we define the flow $f_x:\R\to\MR_{\geq 0}$ on the reaction graph $(\C,\R)$ as follows: on each directed edge $y\to y'$, the flow $f_x(y\to y'):=k_{y\to y'} x^y$. In these terms, Equation~\ref{eqn:ma} can be rewritten as
\begin{equation}\label{eqn:ma2}
\dot{x}(t) = \sum_{y\to y' \in \R} (y' - y) f_{x(t)}(y\to y') 
\end{equation}

A reaction network $(S,\R)$ is \textbf{weakly-reversible} iff every reaction in its reaction graph belongs to a directed cycle. In other words, for every edge $y\to y' \in\R$, there is a path in the graph $(\C,\R)$ from $y'$ to $y$. In other words, every connected component of the reaction graph is strongly-connected.

A mass-action system $(S,\R,k)$ is \textbf{complex-balanced} iff there exists  $\alpha\in \MR^S_{>0}$ such that for every complex $y\in \C$:
\begin{equation}\label{eqn:cmplxblnc}
\sum_{\{ y' \mid y'\to y\in\R\}} k_{y'\to y} \alpha^{y'} = \sum_{\{y''\mid y\to y'' \in\R\}} k_{y\to y''} \alpha^y .
\end{equation}
The point $\alpha$ satisfying Equation~\ref{eqn:cmplxblnc} above is called a {\em point of complex balance}.
In words, when the concentration of species $i\in S$ is $\alpha_i$, the total inflow into vertex $y$ in the reaction graph (LHS of Equation~\ref{eqn:cmplxblnc}) equals the total outflow from that vertex (RHS of Equation~\ref{eqn:cmplxblnc}). Thus $\alpha$ is a point of complex balance iff the flow $f_\alpha$ is {\em conservative}, i.e., at every node, it satisfies {\em Kirchhoff's current law}~\cite{clayton2001fundamentals}.

Recall that a cut is a way of partitioning the vertices of the graph into two sets. Given a conservative flow on a graph, it can be shown by induction that the flow across every cut is zero, by moving one vertex across the cut at a time. We immediately have the following lemma.
\begin{lemma}\label{lem:zerocut}
Let $(S,\R,k)$ be a complex-balanced mass-action system with point of complex balance $\alpha$. Then the net flow $f_\alpha$ across every cut of the reaction graph is zero.
\end{lemma}

It is well-known that if a mass-action system $(S,\R,k)$ is complex-balanced then the reaction network $(S,\R)$ is weakly-reversible. If not, then there are two complexes $y$ and $y'$ connected by a path from $y$ to $y'$ in the reaction graph, with no path from $y'$ to $y$. We can construct a cut consisting of all complexes that have directed paths to $y$ on one side, and everything else on the other. Certainly, $y'$ is on the other side, so this cut is not a trivial cut. Further, this cut has flow from the $y$ side to the $y'$ side, but it can not have any flow coming back. This contradicts Lemma~\ref{lem:zerocut}, and we're done.

In the rest of this paper, $(S,\R)$ will be a weakly-reversible chemical reaction network, $\C$ will be the set of complexes that occur in $\R$, and $(S,\R,k)$ will be a complex-balanced mass-action system with \textbf{point of complex balance} $\alpha\in\MR^S_{>0}$. (It is easily checked that every weakly-reversible chemical reaction network admits extension to a complex-balanced mass-action system: choose $k$ to be the constant $1$ function, and $\alpha$ to be the point of all $1$'s.)

The \textbf{pseudo-Helmholtz function} $g_\alpha : \MR^S_{\geq 0} \to\MR$ of $(S,\R,k)$ at $\alpha$ is the function
\[
g_\alpha(x) = \sum_{i\in S} x_i \log x_i - x_i - x_i \log \alpha_i
\]
where $0 \log 0$ is set to $0$ by definition.  This function was first introduced to the study of mass-action kinetics in the context of complex balance by Horn and Jackson~\cite{HornJackson} in 1972. They also proved the following theorem.

\begin{theorem}\label{thm:main}
If $x(\cdot)$ is a solution to Equation~\ref{eqn:ma} with $x(t)\in\MR^S_{>0}$ for all $t\geq 0$ then $\frac{d g_\alpha(x(t))}{dt} \leq 0$ with equality iff $x(t)$ is a point of complex balance.
\end{theorem}

We present a new and simpler proof to Theorem~\ref{thm:main}. The main idea, following Horn and Jackson, remains to prove the theorem ``one cycle at a time.'' For every cycle, our proof is a straightforward application of the log-sum inequality, which is well-known in information theory and convex analysis. 

\section{Proof for Cycles}
We first illustrate our proof idea with a simple example.
\begin{example}
Consider a single reversible reaction. So $y,y'$ are the two chemical complexes, and the reactions are $y\to y'$ and $y'\to y$. Both the rates are set to $1$. This mass-action system is complex balanced with point of complex balance $\alpha = (1,1)$ because the underlying reaction network $(S, \{y\to y', y'\to y\})$ is weakly-reversible, and for every complex in the reaction graph, all outflows and inflows equal $1$. (More generally, whenever a reaction network is weakly-reversible, and the rate function is constant $1$, the all $1$'s vector is a point of complex balance.) Let $g = g_\alpha$. Let $x(\cdot)$ be a solution to Equation~\ref{eqn:ma}. Then by the chain rule, we have for all $t\geq 0$:
\begin{align*}
\frac{dg(x(t))}{dt} &= \langle \nabla g(x(t)),\dot{x}(t)\rangle\
\\&=\langle \log \frac{x(t)}{\alpha}, k_{y\to y'}(y' - y) x(t)^y + k_{y' \to y}(y-y')x(t)^{y'}\rangle\
\\&= \langle \log \frac{x(t)}{\alpha}, (y' - y) (k_{y\to y'}x(t)^y - k_{y' \to y} x(t)^{y'})\rangle\
\\&= \langle \log x(t), (y' - y) (x(t)^y -  x(t)^{y'})\rangle\
\\&= (x(t)^y -  x(t)^{y'})\log \frac{x(t)^{y'}}{x(t)^y}\
\\&\leq 0.
\end{align*}
Note that $(a - b)\log\frac b a \leq 0$ for all positive reals $a$ and $b$, with equality precisely when $a=b$. So the last inequality is an equality precisely when $x(t)^y = x(t)^{y'}$, which is precisely the condition for $x(t)$ being a point of complex balance in this network.
\end{example}

A reversible reaction is a $2$-cycle. With one extra trick, precisely the above proof works when the reaction graph is a $3$-cycle $y_1 \to y_2 \to y_3$, or even more generally, an $m$-cycle $y_1\to y_2\dots y_m\to y_1$ for distinct complexes $y_1,y_2\dots,y_m$, as we now show. The extra trick is to make use of the log-sum inequality~\cite[Theorem~2.7.1]{cover2012elements} which we now state for convenience.

\begin{theorem}[Log-sum inequality~\cite{cover2012elements}]
For non-negative numbers $a_1,a_2,\dots a_n$ and $b_1, b_2,\dots, b_n$, 
\[
	\sum_{i=1}^n a_i \log \frac{a_i}{b_i} \geq \left(\sum_{i=1}^n a_i\right) \log \frac{\sum_{i=1}^n a_i}{\sum_{i=1}^n b_i}
\]
with equality if and only if $\frac{a_i}{b_i} = \text{const}$.
\end{theorem}

\begin{lemma}\label{lem:cyc}
Let $m\in\MZ_{\geq 2}$. Suppose the reaction graph is an $m$-cycle $y_1\to y_2\to\dots\to y_m\to y_{m+1}=y_1$ where $y_1,y_2\dots,y_m$ are distinct complexes. Then  for all $x\in\MR^S_{> 0}$, we have 
\[
\left\langle \nabla g(x),\sum_{j=1}^m k_{y_j\to y_{j+1}}(y_{j+1} - y_j) x^{y_j} \right\rangle \leq 0
\]
with equality iff $x$ is a point of complex balance.
\end{lemma}
\begin{proof}
Since $f_\alpha$ is a conservative flow, for $j=1$ to $m$, the flow $f_\alpha(y_j\to y_{j+1})$ is constant, independent of $j$. So we may write it as $f_\alpha$. We will also find it convenient to define $M_j(x) := (x/\alpha)^{y_j}$.
\begin{align*}
&\left\langle \nabla g(x),\sum_{j=1}^m k_{y_j\to y_{j+1}}(y_{j+1} - y_j) x^{y_j} \right\rangle\
\\&= \left\langle \log \frac{x}{\alpha}, \sum_{j=1}^m k_{y_j\to y_{j+1}}(y_{j+1} - y_j) x^{y_j} \right\rangle\
\\&= \sum_{j=1}^m \left\langle \log \frac{x}{\alpha}, (y_{j+1} - y_j) k_{y_j\to y_{j+1}}x^{y_j}\right\rangle\
\\&= \sum_{j=1}^m \left\langle \log \frac{x}{\alpha}, (y_{j+1} - y_j) f_\alpha(y_j\to y_{j+1}) \left(\frac{x}{\alpha}\right)^{y_j}\right\rangle\
\\&= f_\alpha \sum_{j=1}^m  \left(\frac{x}{\alpha}\right)^{y_j}\log \frac{(x/\alpha)^{y_{j+1}}}{(x/\alpha)^{y_j}}\
\\&= f_\alpha \sum_{j=1}^m  M_j(x)\log \frac{M_{j+1}(x)}{M_j(x)}\
\end{align*}
Let $M = M_1 + M_2 + \dots + M_m$. Since $f_\alpha$ and all the $M_j$'s are non-negative, by the log sum inequality~\cite[Theorem~2.7.1]{cover2012elements} we can rewrite this last term as: 
\[
f_\alpha M(x) \sum_{j=1}^m  \frac{M_j(x)}{M(x)}\log \frac{M_{j+1}(x)}{M_j(x)} \leq f_\alpha M(x) \log 1 = 0
\]
with equality  iff $M_j(x) = M_{j+1}(x)$ for all $j$ iff $k_{y_j\to y_{j+1}} x^{y_j} = k_{y_{j+1}\to y_{j+2}}x^{y_{j+1}}$ for all $j$ iff $f_x$ is conservative iff $x$ is a point of complex balance. 
\end{proof}

\begin{remark}
Lemma~\ref{lem:cyc} is Lemma~5A in~\cite{HornJackson}. Note that after getting the special form at the end, they have to appeal to a general inequality which they prove in the appendix. 
\end{remark}

\begin{remark}
In fact, the same proof shows that the  statement of Lemma~\ref{lem:cyc} is true even when $x\in\partial\MR^S_{\geq 0}$, provided the value of the logarithm at $0$ is interpreted appropriately according to continuity, and we allow points of complex balance on the boundary also.
\end{remark}

\section{Proof of Theorem~\ref{thm:main}}
In general, for the complex balance system $(S,\R,k)$, the reaction graph need not be a cycle. However, note that the expression $\langle\nabla g_\alpha(x), \sum_{y\to y'\in\R} (y'-y)f_x(y\to y')$ is linear over reactions, so it can be rewritten as 
\begin{equation}\label{eqn:dgdt}
\sum_{y\to y'\in\R}\langle\nabla g_\alpha(x), (y'-y)f_x(y\to y').
\end{equation}
In particular, we will rewrite this sum over directed simple cycles. If the reaction graph consists of cycles that don't share a directed edge --- for example $y_0 \to y_1 \to y_2\to y_1$ and $y_2\to y_1\to y_3\to y_2$ share vertices, but don't share an edge --- then Theorem~\ref{thm:main} follows by linearity, and by invoking Lemma~\ref{lem:cyc} on each cycle. The interesting case is when two cycles share an edge. In this case, we have to partition the flow on that edge across the two cycles, by partitioning the reaction rate across the two cycles. 

\begin{example}
Consider the reaction graph with the two $3$-cycles $C_1 = y_0\to y_1\to y_2\to y_0$ and $C_2 = y_1\to y_2\to y_3\to y_1$. Suppose that $\alpha$ is a point of complex balance for this reaction network. The edge $y_1\to y_2$ is shared between the two cycles. We want two numbers $k^{1}_{y_1\to y_2}$ and $k^{2}_{y_1\to y_2}$ adding up to $k_{y_1\to y_2}$, so that we can treat the sum in Expression~\ref{eqn:dgdt} separately over the two cycles. In other words, we want to define constant flows $f^1_\alpha$ on $C_1$ and $f^2_\alpha$ on $C_2$ such that on $y_1\to y_2$, we have $f_\alpha(y_1\to y_2) = f^1_\alpha + f^2_\alpha$. 

Since $f^1_\alpha$ is to be a constant flow, we define $k^1_{y_1\to y_2}$ from the equation 
\[  k^1_{y_1\to y_2} \alpha^{y_1} = k_{y_0\to y_1} \alpha^{y_0} = f^1_\alpha .\] 
Similarly we define $k^2_{y_1\to y_2}$ from the equation
\[
 k^2_{y_1 \to y_2} \alpha^{y_1} = f_\alpha(y_1\to y_2) - f^1_{\alpha} = f^2_{\alpha}
\]
where the last equality follows since the flow $f_\alpha$ is conservative. Note that $k^1_{y_1\to y_2} + k^2_{y_1\to y_2} = k_{y_1\to y_2}$.

Since we have fixed $k^1_{y_1\to y_2}$ and $k^2_{y_1\to y_2}$, we can define $f^1_x$ and $f^2_x$ for arbitrary $x$ as functions on the edges of $C_1$ and $C_2$ respectively. In particular, 
\[f^1_x(y_1\to y_2) = k^1_{y_1\to y_2} x^{y_1}.\]
We can now write Expression~\ref{eqn:dgdt} as
\[
\sum_{i=1}^2 \sum_{y\to y'\in C_i} \langle\nabla g_\alpha(x), (y'-y)f^i_x(y\to y')\rangle
\]
For each value of $i$, Lemma~\ref{lem:cyc} holds. Hence, we have shown for this example that Expression~\ref{eqn:dgdt} is non-positive, and is zero only if $x$ is a point of complex balance.
\end{example}

The key step in the above example is captured more generally in the following graph theory lemma.
\begin{lemma}\label{lem:graphtheo}
Fix a graph $G$ with a conservative positive flow $f$ on $G$. (This means that $f$ assigns to every directed edge of $G$ a positive real number such that at every node of $G$, the incoming flow equals the outgoing flow.) Then $G$ can be decomposed into simple, directed cycles $C_1,C_2,\dots, C_n$, with constant flows $f^i > 0$ on the $i$'th cycle so that for each edge $y\to y'$ in $G$:
\begin{equation}\label{eqn:flows}
	f(y\to y') = \sum_{i : y\to y' \in C_i} f^i.
\end{equation}
\end{lemma}
\begin{proof}
We will give an algorithm to find such cycles and flows. Consider the flow $f$. Pick an arbitrary directed simple cycle $C_1$ in $G$. Since $G$ is strongly connected, there always exists such a cycle. Assign to this cycle the minimum flow over all edges in the cycle, so $f^1 := \min_{y\to y' \in C_1} f(y\to y')$. Now define a new {\em remainder} flow 
\[
\hat{f}(y\to y') := 
\begin{cases}
f(y\to y') - f^1 \text{ if }y\to y'\in C_1\
\\f(y\to y')\text{ otherwise.}
\end{cases}
\] 
Drop all edges with $\hat{f}(y\to y') = 0$. At least one edge gets dropped, and we obtain a strict subgraph $\hat{G}$ of $G$. It is immediate that the flow $\hat{f}$ is a conservative positive flow on $\hat{G}$, because $\hat{f}$ differs from the conservative flow $f$ by the flow $f^1$ on a cycle, which makes a net contribution of $0$ at every vertex. Continuing in this manner by infinite descent, we will be left with the empty graph, where the lemma holds trivially. 

To see Equation~\ref{eqn:flows}, note that the following is an invariant of the algorithm: if $\hat{f}_k$ represents the remainder flow after $k$ rounds, then for each edge $y\to y'$ in $G$,
\[
f(y\to y') = \hat{f}_k(y\to y') + \sum_{i\leq k : y\to y' \in C_i} f^i.
\]
where $i$ is summed over all the previous rounds. Equation~\ref{eqn:flows} follows since the algorithm terminates with remainder flow zero.
\end{proof}

\begin{remark}
Lemma~\ref{lem:graphtheo} should be compared with Section~6 and, in particular, Lemma~6D in \cite{HornJackson}.
\end{remark}

We are now ready to prove the general case. 
\begin{proof}[Proof of Theorem~\ref{thm:main}]
Fix $t\geq 0$. Let us write $x$ for $x(t)$ and $g$ for $g_\alpha$. Then $\frac{dg_\alpha(x(t))}{dt}$ is given by Expression~\ref{eqn:dgdt}. We can rewrite $f_x(y\to y') = f_\alpha(y\to y')(x/\alpha)^y$. We will prove the inequality for every strongly connected component. By linearity, without loss of generality we may assume that the reaction graph is strongly connected. 

Since $f_\alpha$ is a conservative flow, by Lemma~\ref{lem:graphtheo}, the reaction graph can be decomposed into simple directed cycles $C_1,C_2,\dots,C_n$ and constant flows $f_\alpha^1,f_\alpha^2,\dots,f_\alpha^n$ on them, satisfying Equation~\ref{eqn:flows}. So we may write
\begin{align*}
\frac{dg_\alpha(x(t))}{dt} &=  \sum_{y\to y'\in \R}\left\langle\nabla g_\alpha(x), (y'-y)f_\alpha(y\to y')(x/\alpha)^y\right\rangle\
\\&= \sum_{i=1}^n f_\alpha^i \sum_{y\to y'\in C_i}\left\langle\log(x/\alpha), (y'-y)(x/\alpha)^y\right\rangle\
\end{align*}
By Lemma~\ref{lem:cyc}, each of the inner sums is $\leq 0$, and is zero iff the corresponding flow $f_x^i$ is conservative. The entire sum is $0$ iff each such flow $f_x^i$ is conservative in which case $f_x$ is conservative and $x$ is a point of complex balance. Conversely, if $x$ is a point of complex balance then $f_x$ is conservative and $dx/dt = 0$ forcing $\frac{dg_\alpha(x(t))}{dt}=0$.
\end{proof}\
\\
\\{\em Acknowledgements:} I thank Pranab Sen for pointing out the application of the log-sum inequality in Lemma~\ref{lem:cyc}.
\bibliographystyle{amsplain}
\bibliography{../eventsystems}
\end{document}